\documentclass[12pt]{article}
\usepackage{amsmath,amssymb,amsthm,latexsym,euscript,amscd, authblk}
\usepackage[english]{babel}
\usepackage[english]{babel}
\usepackage[all]{xy}
\usepackage[dvips]{graphics}

\newtheorem{Theorem}{Theorem}
\newtheorem{Proposition}[Theorem]{Proposition}
\newtheorem{Corollary}[Theorem]{Corollary}

\newtheorem{Example}[Theorem]{Example}

\title{Partial abelianization of free product of algebras.}

\author[1]{A.S. Kocherova\thanks {akocherova@ya.ru}}

\author[2, 3]{I.Yu. Zhdanovskiy\thanks {ijdanov@mail.ru}}

\affil[1]{Moscow Institute of Physics and Technology, Russian Federation}
\affil[2]{Moscow Institute of Physics and Technology, Russian Federation, Laboratory AGHA}
\affil[3]{HSE University, Russian Federation}

\begin{document}

\maketitle

\begin{abstract}
In this article we consider partial abelianization of  associative algebra with respect to a subalgebra. This notion is a generalization of usual abelianization of associative algebra and has an application in Quantum Mechanics and Quantum Information Theory. Using combinatorial methods, representation theory and algebraic geometry we study partial abelianization of free product of algebras in our work. 
\end{abstract}

\section{Introduction}

In our article we study some version of abelianization and its applications.
Let $k$ be algebraically closed field of characteristic zero and all algebras are presumed unital and associative.

One of the ideas of Noncommutative Geometry is to generalize duality between affine varieties and ring of algebraic functions onto the case of noncommutative rings. 
Naive idea is to consider abelianization of noncommutative algebra.  Consider noncommutative $k$-algebra $R$. Abelianization of $R$ is a commutative algebra $R^{ab}$ which is a quotient of $R$ by two-sided ideal generated by commutator space $[R,R]$. There are several generalizations of the abelianizations (see for example \cite{Pro}, \cite{KonRos}, \cite{Kap}, \cite{LB}).
We will study so called partial abeianization. Namely, consider pair $R_1 \subset R$, where $R_1$ is subalgebra of $R$. Partial abelianization $(R_1,R)^{ab}$ is a quotient of $R$ by two-sided ideal generated by $[R_1,R_1]$. 
We study partial abelianization of free product $A*B$ of finite-dimensional algebras $A$ and $B$ in our work.

There is a connection of a partial abelianizations with Quantum Mechanics.
Assume that we have two quantum finite-dimensional systems of observables corresponding to algebras $A$ and $B$. Also, assume that there are $m$ simultaneously measured observables of the following type $r_i = a_i + b_i, i = 1,...,m$. It is natural to ask the following question: is there decomposition of the quantum system generated by $A$ and $B$ into more simple parts? Recall that observables are simultaneously measured iff they commutes. Thus, we can reduce our question to the studying of irreducible representations of partial abelianization $(R, A*B)^{ab}$, where $R$ is a subalgebra of $A*B$ generated by elements $r_i, i = 1,...,m$. Our main studying is devoted to the cases $A = k^{\oplus l}$, $B = k^{\oplus 2}$ and $A = k^{\oplus 3}$, $B = k^{\oplus 3}$. 
In the first case algebras $A$ and $B$ have interpretations as algebras generated by observables in quantum systems of particles with spines $(l-1)/2$ and $1/2$ respectively. In the second case $A$ and $B$ correspond to observables in different quantum systems of particles with spine $1$.  

Our work was motivated by the following usage of the partial abelianization in Quantum Information Theory. For this purpose, recall the notion of mutually unbiased bases (briefly, MUBs). This notion plays important role in Quantum Information Theory.
Namely, two orthonormal bases $e_i, i = 1,...,n$ and $f_j, j = 1,...,n$ in n-dimensional vector space $V$ with Hermitian metric $(,)$ is {\it mutually unbiased} if $|(e_i,f_j)|^2 = \frac1n$ for any $i,j = 1,...,n$. Remind the history of the studying of MUBs. 
Firstly, this notion was established by greatest physicist Schwinger \cite{Sch}. 
Also, mathematical version of this notion is a so-called orthogonal decomposition (briefly, OD) of simple Lie algebras $sl(n)$. This object was intensively studied by Thompson, Kostrikin, Ufnarovskii, Pham Huu Tiep and other investigators (\cite{Tho}, \cite{KKU1}, \cite {KP} etc)

Mathematical part of MUBs and ODs lie in the intersection of algebraic geometry, representation theory, combinatorics and mathematical physics. There are many articles about connections of these objects with other branches of math (cf. for example \cite{BLDH}, \cite{BT} \cite{Rus}, \cite{KKU1}, \cite{DEBZ}, \cite{BZ1} and etc.).

Classification of MUBs (and ODs) is very hard problem and completely solved if only $n \le 5$ (\cite{KKU}, \cite{Haa}). It was proved that there is 4-dimensional family of MUBs for $n = 6$ \cite{BZ2}. In the case $n = 7$ one-dimensional family of MUBs was constructed by Petrescu (\cite{Pet}). Assume that $\{e_i\}^7_{i = 1}$ and $\{f_j\}^7_{j = 1}$ are MUBs from the family of Petrescu. Using these bases, Nicoara (\cite{Nic}) constructs two pair of orthogonal projectors $P_1, P_2$ and $Q_1, Q_2$ of rank $2$ such that 
\begin{equation}
\label{nic}
[P_1,Q_1] - [P_2,Q_2] = [P_1 + Q_2, P_2 + Q_1] = 0. 
\end{equation}
Denote by ${\cal N}$ the algebra generated by two pairs orthogonal projectors $P_1,P_2$ and $Q_1,Q_2$ satisfying to relation (\ref{nic}).
Representations of ${\cal N}$ were studied in the article (\cite{KZ1}). This exploration permits to construct Petrescu's family by means of representation theory and algebraic geometry (\cite{KZ2}). One can see that algebra ${\cal N}$ is an example of a partial abelianization of $k^{\oplus 3} * k^{\oplus 3}$.

The main object of exploration is a family of algebras $S_x$ labeled by points of ${\mathbb  P}^3$.
Orthogonal projectors $p_1,p_2$ and $q_1,q_2$ are generators of $S_x$. These projectors satisfy to the relation:
\begin{equation}
\sum^2_{i,j=1} x_{ij}[p_i,q_j] = 0, 
\end{equation}
where $(x_{11}:x_{12}:x_{21}:x_{22}) \in {\mathbb P}^3$.
It is easy that algebras $S_x, x \in {\mathbb P}^3$ are generalization of algebra ${\cal N}$ and one can show that algebras $S_x$ are partial abelianization of free product $k^{\oplus 3} * k^{\oplus 3}$.
Using combinatorial methods, representation theory and algebraic geometry, we get the main theorem of this article:
\begin{Theorem}
For generic $x \in {\mathbb P}^3$ algebra $S_x$ has dimension 18 and isomorphic to direct sum of 9 $k$'s and matrix algebra of order $3$.
\end{Theorem}
This result has the following interpretation in terms of Quantum Mechanics. Assume that we have two particles of spine $1$. Fix two observables $A,B$ in first and second quantum systems respectively. Assume that there are two linear combinations of $A^i, B^j$ which are simultaneously measured. In this case quantum system generated by two particles reduces to direct sum of several copies of one-dimensional and three-dimensional quantum systems.

Our article is organized as follows. Section \ref{secfir} is devoted to properties of partial abelianization. In this section we introduce the following "measure" of noncommutativity of algebra $R$: maximum of dimensions of irreducible representations. We denote it by ${\rm mid}(R)$.  
It is easy that ${\rm mid}(R)$ reflects difference between algebra $R$ and commutative algebras. We show that if $R_1$ sufficiently large subalgebra of $R$ then ${\rm mid}(R_1,R)^{ab} < \infty$. Using this result, we study partial abelianization of $k^{\oplus l}* k^{\oplus 2}$.
In the section \ref{secsec} we study partial abelianization of $k^{\oplus 3}* k^{\oplus 3}$ and prove the main theorem.

We are grateful to A.Bondal, G.Amosov, I.Karzhemanov and Yu.Orlov for helpful discussions.

Second author was partially supported by the HSE University Basic Research Program and the Russian Academic Excellence Project '5-100.
Second author was partially supported by grant RFBR - 18-01-00908

\section{Partial abelianization: definition, properties and examples.}
\label{secfir}

In this section we introduce the notion of partial abelianization of an algebra with respect to its subalgebra. Also, we study properties of the partial abelianization of free product and its irreducible representations. We study the case of $k^{\oplus l} * k^{\oplus 2}$ at the end of section.
\subsection{Partial abelianization and $MID_r$ - algebras.}

In this section we will make some remarks on the properties of abelianization and $mid$ - algebras. 

Let $R$ be a unital associative algebra over $k$. Recall that its abelianization $R^{ab}$ is a quotient of $R$ by two-sided ideal generated by the complete commutator subspace $[R,R]$. It is well-known that abelianization is a functor from category of associative algebras to category of commutative algebras.

Define partial(or relative) abelianization as follows. Consider a category ${\cal C}$ of the pairs of algebras $(R_1, R)$ with fixed morphism of algebras $f:R_1 \to R$. Fix two pairs $(R_1,R)$ and $(R'_1, R')$ and morphisms $f: R_1 \to R$ and $f':R'_1 \to R'$. Define a morphism $g$ in the category ${\cal C}$ as follows: $g: (R_1,R) \to (R'_1,R')$ is a pair of morphisms $g_1:R_1 \to R'_1$ and $g_2:R \to R'$ compatible with the morphisms $f$ and $f'$. Denote by ${\rm Assoc}_k$ the category of associative algebras over $k$. 
We will call the free product $(R_1,R)^{ab} = R*_{R_1} R^{ab}_1$ by {\it partial(or relative)} abelianization. Note the following trivial properties of the partial abelianization:
\begin{itemize}
\item{For any $R_1$ and morphism $f$ there is a surjective morphism: $R \to (R_1,R)^{ab}$.}
\item{If $f: R_1 \to R$ is surjective then $(R_1,R)^{ab} = R^{ab}$ (i.e. a usual abelianization).}
\item{There is a natural morphism of algebras: $R^{ab}_1 \to (R_1,R)^{ab}$. If $f$ is a monomorphism then $R^{ab}_1$ is a subalgebra of $(R_1,R)^{ab}$.}
\item{If $f(R_1)$ is a commutative subalgebra of $R$, then $(R_1,R)^{ab} = R$.} 
\item{Partial abelianization: $(R_1,R)^{ab}$ is a quotient of $R$ by the two-sided ideal generated by $[f(R_1),f(R_1)]$.}
\item{If $f$ is a monomorphism, then there is a surjective morphism $(R_1,R)^{ab} \to R^{ab}$.} 
\end{itemize}
It is easy to see that partial abelianization plays an intermediate role between $R$ and $R^{ab}$. 
Denote by ${\rm mid}(R)$ ( the maximal dimension of irreducible representation of a fixed algebra $R$. ${\rm mid}(R)$ "measures" noncommutativity of $R$. 
We will call algebra $R$ by $MID_r$ - algebra if ${\rm mid}R \le r$. If ${\rm mid}R = \infty$ then algebra $R$ is called $MID_{\infty}$ - algebra.
It is easy that if $R$ is commutative then ${\rm mid}R = 1$. If $R$ is a free algebra with more than one generator then ${\rm mid}(R) = \infty$. One can construct many examples of the algebra $R$ such that ${\rm mid}(R) = \infty$. It is natural to consider algebras $R$ satisfying the condition ${\rm mid}(R) \le r$ for some arbitrary $r$. 


Consider some condition on $R_1,R$ such that $(R_1,R)^{ab}$ is $MID_r$ - algebra for some $r$.
Firstly, recall the property of $MID_r$ - algebra.
Note the following useful property of $MID_r$ - algebras.
\begin{Proposition}
\begin{itemize}
\item{Consider an algebra $R$. If some quotient $R/I$ is not $MID_r$ - algebra for some $r$, then $R$ is not $MID_r$ - algebra.} 
\item{Assume $R_1 \subset R$ is a subalgebra and $R_1$ is a $MID_r$ - algebra. If $R/R_1$ is a  projective $R_1$ - module, then $R$ is $MID_k$ - algebra for some $k \ge r$ ($k$ may be $\infty$).} 
\end{itemize}
\end{Proposition}
\begin{proof}
First statement is trivial. Fix irreducible $R_1$-module $V$.
Using the following exact sequence of $R_1$ - bimodules,
\begin{equation}
\label{exseqbc}
\xymatrix{
0 \ar[r] & R_1 \ar[r] & R\ar[r] & R/R_1 \ar[r] & 0,
}
\end{equation}
we obtain that natural morphism of $R_1$-modules: $p:{\rm Hom}_{R_1}(R,V) \to V$ is surjective. It is easy to see that ${\rm Hom}_{R_1}(R,V)$ is $R$-module. Let $W$ be $R$-submodule of ${\rm Hom}_{R_1}(R,V)$. Of course $W$ is $R_1$ - module. Since $V$ is irreducible $R_1$-module, we have two possibilities: $p(W) = 0$ or $p(W) = V$. Taking ${\rm Hom}_{R_1}(R,V) / W$ or $W$ enough times, we get irreducible $R$-module $W'$ with surjective morphism of $R_1$-modules: $W' \to V$. Thus, we conclude the proof.  
\end{proof}

Also, we have the following property of $MID_r$ - algebras:
\begin{Proposition}
\label{comm}
Consider an algebra $R$ with subalgebra $R_1$. Assume that there is a surjective morphism of $R_1$-modules: $R^{\oplus s}_1 \to R$ for arbitrary $s$. If $R_1$ is a $MID_r$ - algebra, then $R$ is a $MID_{sr}$ - algebra. 
\end{Proposition}
\begin{proof}
Fix irreducible $R$ - module $V$. If we consider $V$ as $R_1$ - module, then we have $R_1$ - submodule $M \subseteq V$ such that ${\rm dim}_{\mathbb F}M \le r$. Using adjunction of functors, we get the following isomorphism: ${\rm Hom}_{R_1}(M,V) \cong {\rm Hom}_{R}(R \otimes_S M, V)$. Since ${\rm Hom}_{R_1}(M,V) \ne 0$, then there is a non-trivial morphism: $R \otimes_{R_1} M \to V$. Since $V$ is irreducible, we obtain that this morphism is surjective. Using surjectivity of morphism: $R^{\oplus s}_1 \to R$, we get that ${\rm dim}_k V \le sr$.
\end{proof}

\begin{Corollary}
Let $R_1$ be a subalgebra of $R$. Assume that $R^{ab}_1$ and $(R_1,R)^{ab}$ satisfy the condition of proposition \ref{comm} for some $s$. Then $(R_1,R)^{ab}$ is a $MID_s$ - algebra.
\end{Corollary}

Note the following property of  representations of partial abelianization $S = (R_1,R)^{ab} = R *_{R_1} R^{ab}_1$, where $R_1$ is a subalgebra of $R$.
Recall that there is a well-defined functor ${\bf Rep}: Assoc_k \to Sets$. This functor is defined by the formula: ${\bf Rep}(A) = {\rm Hom}_{Assoc_k}(A, M_n(k))$ for a fixed algebra $A$. If $A$ is finite-generated, then ${\bf Rep}_n(A)$ is an affine scheme over $k$. 
We have the following isomorphism: ${\bf Rep}_n(S) = X \times_{Y} Y'$, where $X = {\bf Rep}_n(R)$, $Y = {\bf Rep}_n(R_1)$ and $Y' = {\bf Rep}_n(R^{ab}_1)$.



\subsection{Free product and partial abelianization.}
In this section we will study partial abelianization of the free product of finite-dimensional algebras. 

Let $A$ and $B$ be the finite-dimensional algebras. Consider free product $A*B$. It is easy to see that representation theory of algebra $A*B$ is more complicated than that of $A$ and $B$. For example, if $A$ and $B$ are semisimple algebras (i.e. the global homological dimension of $A$ and $B$ is zero) then $A*B$ is quasi-free, i.e. of homological dimension 1. Also, it is easy to see that ${\rm mid}(A*B) \ge \max({\rm mid}(A), {\rm mid}(B))$. Consider following examples.

\begin{Example}
\label{exfir}
Consider an algebra $R = k^{\oplus 2} * k^{\oplus 2}$. First and second $k^{\oplus 2}$ are generated by idempotents $p,q$ respectively. It is well-known that $R$ is $MID_2$-algebra. Actually, this algebra has a non-trivial center $Z(R)$ generated by the element $z = -p-q+pq+qp$. It can be shown that $R$ is free $Z(R)$ - module of rank 4. Thus, ${\rm mid}(R) = 2$.
\end{Example}

\begin{Example}
\label{exsec}
Consider an algebra $R = k^{\oplus s}* k^{\oplus m} \cong k[{\mathbb Z}_s * {\mathbb Z}_l]$ where $s \ge l \ge 2$ and $s \ge 3$. Using Kurosh's Theorem, we get that free group $F$ of rank $(s-1)(l-1)$ is a subgroup of $G = {\mathbb Z}_s * {\mathbb Z}_l$ of finite index. 
It is easy to show that $k[G]$ as $k[F]$ - module has the following decomposition: $k[G] = \bigoplus_{x \in G/F} x k[F]$. Thus, $k[G]/k[F]$ is a free $k[F]$ - module. Note that ${\rm mid}(k[F]) = \infty$ for a group algebra of free group of rank more than 1. Applying proposition \ref{comm}, we get that algebra $R$ is $MID_{\infty}$-algebra.
\end{Example}

We will study partial abelianization of $A*B$ of special type.
Fix the decomposition of vector spaces: $A = k \cdot 1 \oplus \overline{A}$ and $B = k \cdot 1 \oplus \overline{B}$. Consider a subspace $V \subset \overline{A} \oplus \overline{B} \subset A*B$.
Denote by $R(V)$ the subalgebra generated by subspace $V$. Consider the pair $(R(V),A*B)$ and partial abelianization $(R(V),A*B)^{ab}$. For simplicity, denote they by $(V,A*B)$ and $(V,A*B)^{ab}$ respectively. Consider the variety ${\rm Gr}(r,\overline{A} \oplus \overline{B})$ of $r$ - dimensional subspaces of $\overline{A} \oplus \overline{B})$. 

Let us formulate the following useful proposition:
\begin{Proposition}
\label{fmod}
Let $A$,$B$ be finite-dimensional algebras. Let $V$ be a subspace of $\overline{A} \oplus \overline{B}$ such that the restrictions of natural projections $\pi_A|_V:V \subset \overline{A} \oplus \overline{B} \to \overline{A}$ and $\pi_B|_V: V \subset \overline{A} \oplus \overline{B} \to \overline{B}$ are surjective. Consider a subalgebra $R(V) \subset A*B$ generated by $V$. Put $m = \min({\rm dim}_kA, {\rm dim}_kB)$. Then we have the following surjective morphism of $R(V)$ - modules:
$R(V)^{\oplus m} \to A*B$. 
\end{Proposition}
\begin{proof}
Assume that ${\rm dim}_k A \ge {\rm dim}_k B$. 
It is obvious that basis of $V$ has the following view:
$v_i = a_i + b_i, a_i \in \overline{A}, b_i \in \overline{B}$, where $a_i$ and $b_i$ generate $\overline{A}$ and $\overline{B}$ as linear spaces respectively. 
Since $a_i = v_i - b_i$, elements $v_i$ and $b_j$ are generators of algebra $A*B$.

We have the following identities $a_i \cdot a_j = \alpha_{ij0} \cdot 1 + \sum_s \alpha_{ijs}a_s$ and $b_i \cdot b_j = \beta_{ij0} \cdot 1 + \sum_t \beta_{ijt} b_t$, where $\alpha_{ijs}, \beta_{ijt} \in k$ are structural coefficients of algebras $A$ and $B$ respectively. Thus, we have the following identities:
\begin{equation}
(v_i - b_i)\cdot (v_j - b_j) = \alpha_{ij0} \cdot 1 + \sum_s \alpha_{ijs}(v_s - b_s) , i,j = 1,...,{\rm dim}A - 1,
\end{equation}
\begin{equation}
\label{bivj}
-b_i \cdot v_j = v_i \cdot b_j - v_i \cdot v_j + \sum_s \alpha_{ijs} v_s  - \sum_t \beta_{ijt} b_t + (\alpha_{ij0} - \beta_{ij0}) \cdot 1  - \sum_s \alpha_{ijs}b_s.
\end{equation}
In this formula, $ - v_i \cdot v_j  + \sum_s \alpha_{ijs} v_s \in R(V)$, $- \sum_t \beta_{ijt} b_t + (\alpha_{ij0} - \beta_{ij0}) \cdot 1  - \sum_s \alpha_{ijs}b_s \in B$. 
Any element of $A*B$ can be written as linear combination of monomials: $a_{i_1}b_{j_1}...a_{i_s}b_{j_p}$, $a_{i_1}b_{j_1}...a_{i_p}b_{j_p}a_{i_{p+1}}$, $b_{i_1}a_{j_1}...b_{j_p}a_{j_p}$ and $b_{i_1}a_{j_1}...b_{j_p}a_{j_p}b_{j_{p+1}}$. Make the substitution $a_i = v_i - b_i$. Thus, any element of $A*B$ can be written as a linear combination of monomials: $v_{i_1}b_{j_1}...v_{i_p}b_{j_p}$, $v_{i_1}b_{j_1}...v_{i_p}b_{j_p}v_{i_{p+1}}$, $b_{i_1}v_{j_1}...b_{j_p}v_{j_p}$ and $b_{i_1}v_{j_1}...b_{j_p}v_{j_p}b_{j_{p+1}}$. Using formula (\ref{bivj}) and induction, we get that any element $x \in A*B$ can be written in the following manner: $x = \sum^{{\rm dim}B}_{t=1} r_t b_t$ , where $r_t \in R(V)$, $b_t \in B$. 
\end{proof}

We will say that a subspace $V$ is {\it generic with respect to $\overline{A}$ and $\overline{B}$} if $V$ satisfies to the condition of proposition \ref{fmod}. 
Applying proposition \ref{comm}, we get the following:
\begin{Corollary}
\label{commfree}
Consider a free product $A*B$. Put $m = \min({\rm dim}_kA, {\rm dim}_k B)$. Let $V$ be a generic subspace with respect to $\overline{A}$ and $\overline{B}$. Then ${\rm mid} ((V,A*B)^{ab}) \le m$, i.e. $(V,A*B)^{ab}$ is a $MID_m$ - algebra.
\end{Corollary}

\subsection{Partial abelianization of algebra $k^{\oplus l} * k^{\oplus 2}$.}

In this subsection we study partial abelianizations of type $(V,R)^{ab}$, where $R = A*B$, $A = k^{\oplus l}$, $B = k^{\oplus 2}$ and $V \subset \overline{A} \oplus \overline{B}$.

Let $p_1,...,p_l$ and $q_1,q_2$ be sets of orthogonal primitive idempotents of $A$ and $B$ respectively.
It is easy to see that $\sum^l_{i=1}p_1 = q_1+q_2 = 1$. Idempotents $p_1,...,p_{l-1}$ and $q_1$ linearly generate subspaces $\overline{A}$ and $\overline{B}$ respectively.
Consider $V \subset \overline{A} \oplus \overline{B}$. 
We have the following map: ${\rm Gr}(r, \overline{A} \oplus \overline{B}) \setminus {\rm Gr}(r,\overline{A}) \to {\rm Gr}(r-1, \overline{A})$ defined by rule: $V \mapsto V \cap \overline{A}$. 
Denote by $A_1 \subseteq A$ the subalgebra generated by $V_1 = V \cap \overline{A}$. It is easy that algebra $A(V)$ lies in the center of $(V,R)^{ab}$

Any subalgebra of $A$ is isomorphic to $k^{\oplus s}$ for arbitrary $s$. A fixed subalgebra $A' \subset A$ gives us the decomposition of $\{1,...,l\}$ into the union of non-intersecting sets: $\{1,...,l\} = \cup^s_{j=1}I_j$. Also, $A_1$ has a basis $P_j = \sum_{i \in I_j}p_i, j = 1,...,s$. Note that $\sum^s_{j=1}P_j = 1$.

We have the trivial description of partial abelianizations $(V,R)^{ab}$ by geometric properties of $V$ as follows.
Let $\{i_1,...,i_k\}$ be a subset of $\{1,...,l\}$.
Denote by $W_{i_1,...,i_k}, k \ge 2$ the subspace of $A$ given by $x_{i_1} = x_{i_2} = ... = x_{i_k}$, where $x_i$ are the coordinates with respect to the basis $p_i, i = 1,...,l$.
For any subspace $V_1 \subset \overline{A}$  we have the following presentation of $V_1$: 
$$
V_1 \subseteq \bigcap^s_{j=1}W_{I_j},
$$
where $I_j, j = 1,...,s$ ($s$ may be zero) are non-intersecting subsets of $\{1,...,l\}$. 
Also, assume that this presentation of $V_1$ is minimal. 
In this case $A_1 \cong k^{\oplus s}$ and $A_1$ has the basis $P_j = \sum_{t \in I_j}p_k$.

Let us formulate the following proposition:
\begin{Proposition}
\begin{itemize}
    \item{If $V \subseteq \overline{A}$ then $(V,R)^{ab} = R$} 
    \item{If $V \nsubseteq \overline{A}$ and $A(V) = A$ then $(V,R)^{ab} \cong A \otimes B$, and hence, ${\rm mid}((V,R)^{ab}) = 1$.}
    \item{If $V \nsubseteq \overline{A}$ and $\max_j\{|I_j|\} = 2$ then ${\rm mid}((V,R)^{ab}) = 2$.}
    \item{If $V \nsubseteq \overline{A}$ and $\max_j \{|I_j|\} > 2$ then ${\rm mid}((V,R)^{ab}) = \infty$.}
\end{itemize}
\end{Proposition}
\begin{proof}
First and second statements are trivial.
It is easy to see that irreducible $(V,R)^{ab}$ -module defines the character of $A(V)$. 
One can check that $(V,R)^{ab} \otimes_{A(V)} k^{\chi} \cong k^{\oplus |I_j|} * k^{\oplus 2}$ for $A(V)$ - character $\chi$. 
Applying examples \ref{exfir} and \ref{exsec}, we get third and fourth statements.
\end{proof}

\section{Partial abelianization of $k^{\oplus 3}*k^{\oplus 3}$.}
\label{secsec}

In this section we study the partial abelianization of $R = A*B = k^{\oplus 3} * k^{\oplus 3}$ with respect to the subalgebra generated by generic $V \subset \overline{A} \oplus \overline{B}$.

\subsection{Preliminary remarks.}
In this section we demonstrate some geometry related to partial abelianization.

Consider a variety ${\rm Gr} = {\rm Gr}(2,\overline{A} \oplus \overline{B})$
parameterizing two-dimensional subspaces in $\overline{A} \oplus \overline{B}$.
We have the following mapping:
$f: {\rm Gr} \dasharrow {\mathbb P}([\overline{A}, \overline{B}])$ defined by rule: 
$V \mapsto [V,V] \in {\mathbb P}([\overline{A}, \overline{B}])$.
It is easy to see that the mapping $f$ is not well-defined in two points corresponding to subspaces $\overline{A}$ and $\overline{B}$. Clearly, $(\overline{A},R)^{ab} = R$ and $(\overline{B},R)^{ab} = R$.

Mapping $f$ is a composition of Plucker embedding ${\rm Gr} \to {\mathbb P}^5 = {\mathbb P} (\Lambda^2(\overline{A} \oplus \overline{B}))$ and the projection ${\mathbb P} (\Lambda^2(\overline{A} \oplus \overline{B})) \setminus L \to {\mathbb P}([\overline{A} \oplus \overline{B}])$, where $L$ is a line in ${\mathbb P}^5$ intersect ${\rm Gr}$ in two points corresponding to subspaces $\overline{A}$ and $\overline{B}$.
It is easy to see that $f$ is surjective. 
For any $x \in {\mathbb P}^3 = {\mathbb P}([\overline{A}, \overline{B}])$ pre-image $f^{-1}(x)$ is isomorphic to conic without two points.
Let $p_i, i = 1,2,3$ and $q_i, i = 1,2,3$ be the sets of orthogonal primitive idempotents of $A$ and $B$ respectively. Let $p_1,p_2$ and $q_1,q_2$ are bases of $\overline{A}$ and $\overline{B}$ respectively.
Basis of $[\overline{A},\overline{B}]$ is $[p_i,q_j], i,j = 1,2$. Denote by $x_{ij}, i,j = 1,2$ the coordinates of ${\mathbb P}^3$ with respect to the basis $[p_i,q_j]$. Since $f$ is surjective, we get that any abelianization $(V,R)^{ab}, V \in {\rm Gr}, V \ne \overline{A}, \overline{B}$  is a quotient of $k^{\oplus 3}*k^{\oplus 3}$ by relationship:
\begin{equation}
\label{commrel}
\sum^2_{i,j=1}x_{ij}[p_i,q_j] = 0    
\end{equation}
for an arbitrary $x = (x_{11}:x_{12}:x_{21}:x_{22}) \in {\mathbb P}^3$.
Describe the map $f$ in coordinates.
Consider Zariski-open subset $U = {\mathbb C}^4 \subset {\rm Gr}$ consisting of subspaces which have the basis of the following type: $p_1 - y_2q_1-y_3q_2$ and $p_2+y_0q_1+y_1q_2$. 
It is obvious that $[p_1-y_2q_1-y_3q_2,p_2+y_0q_1+y_1q_2] = y_0[p_1,q_1]+y_1[p_1,q_2]+y_2[p_2,q_1]+y_3[p_2,q_2]$.
Thus, $f|_U: U \dasharrow {\mathbb P}^3$ is given by the formula: $(y_0,y_1,y_2,y_3) \mapsto (y_0:y_1:y_2:y_3)$.


It is easy $(V,R)^{ab}, V \ne \overline{A}, \overline{B}$ is a sheaf of algebras over ${\mathbb P}^3$.
Denote by $S_x, x \in {\mathbb P}^3$ the algebra $(V,R)^{ab}, f(V) = x$.

Further, consider $V \in {\rm Gr}$ which does not satisfies to the conditions of corollary \ref{commfree}, i.e. the restriction of $\pi_A$ (or $\pi_B$) to $V$ is not a surjective morphism. In this case we get that $V \cap \overline{A} \ne 0$ or $V \cap \overline{B} \ne 0$.
Denote by $Z_A$ and $Z_B$ the divisors of ${\rm Gr}$ consisting of $V$ such that $V \cap \overline{A} \ne 0$ and $V \cap \overline{B} \ne 0$ respectively.
\begin{Proposition}
Consider quadric $Q \subset {\mathbb P}^3$ given by the equation: $x_{11}x_{22} = x_{12}x_{21}$. We have the following statements:
\begin{itemize}
    \item{$f(Z_A) = f(Z_B) = Q$ and $f^{-1}(Q) = Z_A \cup Z_B$} 
    \item{If $x$ is a generic point of $Q$ then the relation (\ref{commrel}) can be written in the following manner: $[a,b] = 0$ for some $a \in \overline {A}$ and $b \in \overline{B}$}
\end{itemize}
\end{Proposition}
\begin{proof}
By direct calculations. 
\end{proof}

Denote by $l_i \subset Q \subset {\mathbb P}^3 , i =1,2,3$ the lines given by the equations:
$x_{21} = x_{22} = 0$, $x_{11} = x_{12} = 0$ and $x_{11} = x_{21}, x_{12} = x_{22}$ respectively.
Also, denote by $l'_i \subset Q \subset {\mathbb P}^3, i = 1,2,3$ the lines given by
$x_{12} = x_{22} = 0$, $x_{11} = x_{21} = 0$ and $x_{11} = x_{12}, x_{21} = x_{22}$, respectively.
Denote by $pt_{ij} = l_i \cap l'_j, i,j = 1,2,3$.

Thus, we deduce the following corollary:
\begin{Corollary}
We have the following classification of $S_x, x \in Q$ (or $(V,R)^{ab}$ for $f(V) = x \in Q$):
\begin{itemize}
    \item{if $x = pt_{ij}, i,j = 1,2,3$ then $S_x$ is isomorphic to quotient $k^{\oplus 3}*k^{\oplus 3}$ by the relation $[p_{i_0},q_{j_0}] = 0$ for arbitrary $i_0$ and $j_0$. Moreover, ${\rm mid}(S_x) = \infty$.} 
    \item{if $x \in (\bigcup^3_{i=1} l_i \cup l'_i) \setminus \bigcup^{3}_{i,j} pt_{ij}$, then $S_x$ is isomorphic to the quotient of $k^{\oplus 3}*k^{\oplus 3}$ by the relations: $[p_i,q_{j_0}] = 0, i = 1,2$ for fixed $j_0$ or by relations $[p_{i_0},q_j] = 0, j = 1,2$ for fixed $i_0$. Moreover, ${\rm mid}(S_x) = \infty$.}
    \item{if $x \in Q \setminus (\bigcup^3_{i=1} l_i \cup l'_i)$, then $S_x$ is isomorphic to $k^{\oplus 9}$ and, hence, ${\rm mid}(S_x) = 1$.}
\end{itemize}
\end{Corollary}
\begin{proof}

In the first case consider the quotient of $S_x$ by the ideal generated by $p_{i_0}$. We get that surjective morphism: $S_x \to k^{\oplus 2} * k^{\oplus 3}$. Using example \ref{exsec}, we get that ${\rm mid}(S_x) = \infty$. Analogous arguments show us that ${\rm mid}(S_x) = \infty$ in the second case.
Using the standard arguments, we get that $[a,b] = 0$ for any $a \in A$ and $b \in B$, and hence, ${\rm mid}(S_x) = 1$.

\end{proof}

\subsection{A combinatorial description of partial abelianization.}

In this subsection we study algebra $S_x$ for generic $x \in {\mathbb P}^3$ by combinatorial methods. 

Assume that $V \subset \overline{A} \oplus \overline{B}$ is a generic subspace.
Denote by ${\mathbb F}$ the field $k(y_1,y_2,y_3)$,  ${\cal A}$ and ${\cal B}$ the ${\mathbb F}$-algebras $A \otimes_k{\mathbb F}$ and $B \otimes_k{\mathbb F}$ respectively.
Also, denote by ${\cal S}$ the unital ${\mathbb F}$-algebra which is a quotient of ${\cal R} = R \otimes_k {\mathbb F} = {\cal A} * {\cal B}$ (here the free product is given over $\mathbb F \cdot 1$) by the relation:
\begin{equation}
\label{relx}
X = [p_1,q_1] + y_1[p_1,q_2] + y_2[p_2,q_1]+y_3[p_2,q_2] = 0.
\end{equation}
In this subsection we get an estimation of the dimension of $\cal S$ over ${\mathbb F}$.

Obviously, $1,p_1,p_2$ and $1,q_1,q_2$ are the linear bases of ${\cal A}$ and ${\cal B}$ respectively.
It is easy to see that there is a ${\mathbb F}$ - basis of ${\cal R}$ of the following type: $p_{i_1}q_{j_1}...p_{i_s}$, $q_{i_1}p_{j_1}...q_{j_s}$, $p_{i_1}q_{j_1}...p_{i_s}q_{j_s}$ and $q_{i_1}p_{j_1}...q_{i_s}p_{j_s}$ where $s$ runs over all integers.
Thus, one can define the 'length' function of ${\cal R}$ by obvious way.
Also, there is a filtration $F$ on ${\cal R}$ given by:
\begin{itemize}
    \item{$F^0 {\cal R} = {\mathbb F}\cdot 1$}
    \item{$F^s {\cal R}$ is a vector space with basis of elements of length less or equal $s$.}
\end{itemize}
One can check that $F^i{\cal R} \cdot F^j{\cal R} \subseteq F^{i+j}{\cal R}$ for any $i,j$.
It is easy that ${\rm dim}_{\mathbb F}F^k{\cal R} = 2^{k+2} - 3$.
Denote by $I(X)$ the ideal generated by the element $X \in F^2{\cal R}$ from (\ref{relx}).
There is a natural filtration $F$ on algebra ${\cal S}$ induced by the filtration on ${\cal R}$.

Consider 13 elements: $X, p_iX, Xp_i, q_iX, Xq_i, p_iXp_j, q_iXq_j$ for $i,j = 1,2, i \ne j$ of length less or equal to $3$ and 40 elements of length $4$: 
$p_i q_j X, p_i X q_j, X p_i q_j, q_i p_j X, q_i X p_j, X q_i p_j, i,j = 1,2$, 
$p_i X p_j q_k, q_k p_i X p_j, q_i X q_j p_k, p_k q_i X q_j, i\ne j, i,j,k = 1,2$. 
Denote by $N \subset F^4 I(X)$ the space generated by these elements.
Using computations on Maple, we get the following:
\begin{Proposition}
${\rm dim}_{\mathbb F} N = 42$. Thus, ${\rm dim}_{\mathbb F}F^4 I(X) \ge 42$ and, hence, ${\rm dim}_{\mathbb F}F^4 {\cal S} \le 19$.
\end{Proposition}

There exist the following linear generators of $F^4 {\cal S}$:
\begin{itemize}
\item{$1,p_1,p_2,q_1,q_2$}
\item{$p_1q_2,p_2q_1,p_2q_2,q_1p_1,q_1p_2,q_2p_1,q_2p_2$}
\item{$p_2q_1p_2,p_2q_2p_1,q_1p_1q_1,q_1p_2q_2,q_2p_2q_2$}
\item{$p_2q_1p_1q_1,q_2p_1q_1p_1$}
\end{itemize}
(these elements may be linearly dependent). 
Also, the calculations on Maple get us the following identities: 
\begin{equation}
\label{pq}
p_iq_jp_kq_l = \alpha(i,j,k,l) \cdot p_2q_1p_1q_1 + w_1(i,j,k,l)    
\end{equation}
and 
\begin{equation}
\label{qp}
q_ip_jq_kp_l = \beta(i,j,k,l) \cdot q_2p_1q_1p_1 + w_2(i,j,k,l)    
\end{equation}
for $\alpha(i,j,k,l), \beta(j,i,k,l) \in {\mathbb F}$ and $w_1(i,j,k,l), w_2(i,j,k,l) \in F^3 {\cal S}$.

Consider an element $z = p_{i_1}q_{j_1}p_{i_2}q_{j_2}p_{i_3}$. Using the identities (\ref{pq}) and (\ref{qp}), we obtain that
$$
z = p_{i_1} \cdot q_{j_1}p_{i_2}q_{j_2}p_{i_3} = \beta(j_1,i_2,j_2,i_3) (p_{i_1}q_2p_1q_1)p_1 + p_{i_1}w_2(j_1,i_2,j_2,i_3) = 
$$
$$
\alpha(i_1,2,1,1) \beta(j_1,i_2,j_2,i_3) p_2q_1p_1q_1p_1 + w_1(i_1,2,1,1) p_1 + p_{i_1}w_2(j_1,i_2,j_2,i_3),
$$
$w_1(i_1,2,1,1)p_1, p_{i_1}w_2(j_1,i_2,j_2,i_3) \in F^4 {\cal S}$. Analogously, one can show that 
$$
q_{i_1}p_{j_1}q_{i_2}p_{j_2}q_{i_3} = \beta(i_1,2,1,1)\alpha(j_1,i_2,j_2,i_3) q_2p_1q_1p_1q_1 + w',
$$  
where $w' \in F^4 {\cal S}$. 
Thus, $F^5{\cal S} / F^4 {\cal S}$ is generated by $p_2q_1p_1q_1p_1$ and $q_2p_1q_1p_1q_1$ (they may be linearly dependent). 

Further, consider element $p_2q_1p_1q_1p_1$. We have the following identity: $p_2 \cdot q_1p_1q_1p_1 = \beta(1,1,1,1)  p_2q_2p_1q_1p_1 + w' = \alpha(2,2,1,1) \beta(1,1,1,1) p_2q_1p_1q_1p_1 + w''$, where $w',w'' \in F^4 {\cal S}$. Direct calculations show us that $\alpha(2,2,1,1) \beta(1,1,1,1) \ne 1$. Thus, $p_2q_1p_1q_1p_1 \in F^4 {\cal S}$. Analogously, $q_2p_1q_1p_1q_1 \in F^4 {\cal S}$. Hence, $F^i {\cal S} / F^4 {\cal S} = 0$ for $i \ge 5$. 

We have proved the following statement:
\begin{Proposition}
\label{estdim}
Algebra $\cal S$ is a finite-dimensional ${\mathbb F}$ - algebra. Also, we have the following inequality:
\begin{equation}
{\rm dim}_{\mathbb F} {\cal S} \le 19.
\end{equation}
\end{Proposition}

We have a well-defined action of symmetric group $\Sigma_3$ on ${\cal R}$ by permutations of $p_1,p_2$ and $1-p_1-p_2$.
One can define an action of $\Sigma_3$ on $p_i$ and ${\mathbb F}$ so that these permutations preserve the ideal $I(X)$. 
Denote by $\sigma_1$ the transposition permuting $p_1$ and $p_2$. Define the action of $\sigma_1$ on ${\mathbb F}$ by the formula: 
\begin{equation}
\label{sf}
\sigma_1: y_1 \mapsto \frac{y_3}{y_2}, y_2 \mapsto \frac{1}{y_2}, y_3 \mapsto \frac{y_1}{y_2}.    
\end{equation}
One can check that $\sigma_1(X) = \frac{1}{y_2} X$.
$\sigma_2$ is a transposition permuting $p_1$ and $1-p_1-p_2$. Define the action of this permutation on ${\mathbb F}$ by the formula: 
\begin{equation}
\label{ss}
\sigma_2: y_1 \mapsto y_1, y_2 \mapsto 1-y_2, y_3 \mapsto y_1-y_3.
\end{equation}
We get that $\sigma_2(X) = -X$. One can check that the formulas (\ref{sf}) and (\ref{ss}) define the well-defined action of $\Sigma_3$ on ${\mathbb F}$. 

Consider an invariant field ${\mathbb F}_1 = {\mathbb F}^{\Sigma_3}$. Thus, we have the well-defined action of $\Sigma_3$ on algebra ${\cal S}$ by the automorphisms of ${\mathbb F}_1$-algebras. We have the following decomposition of the algebra ${\cal S}$ into a direct sum: ${\cal S} = p_1{\cal S} \oplus p_2{\cal S} \oplus (1-p_1-p_2){\cal S}$. Using the action of $\Sigma_3$, we get that the dimensions of subspaces $p_1{\cal S}$, $p_2{\cal S}$ and $(1-p_1-p_2){\cal S}$ over ${\mathbb F}_1$ are the same. Since these subspaces are vector spaces over ${\mathbb F}$, we get that ${\rm dim}_{{\mathbb F}_1}p_1{\cal S} = [{\mathbb F}:{\mathbb F}_1] \cdot m = 6m$, where $m = {\rm dim}_{\mathbb F}p_1{\cal S}$. Thus, we get that ${\rm dim}_{{\mathbb F}_1}{\cal S} = 18m$ and , hence, ${\rm dim}_{\mathbb F}{\cal S} = 3m$.
Thus, we have proved the following statement:
\begin{Proposition}
The dimension of ${\cal S}$ as a vector space over ${\mathbb F}$ is divisible by 3.
\end{Proposition}

Using the proposition \ref{estdim}, we obtain that
\begin{Corollary}
\label{esc}
${\rm dim}_{\mathbb F}{\cal S} \le 18$.
\end{Corollary}

\begin{Corollary}
${\rm dim}_k S_x \le 18$ for general point $x \in {\mathbb P}^3$.
\end{Corollary}

Note that if $x = (1:0:0:-1)$ then the algebra $S_x$ is infinite-dimensional (cf \cite{KZ1}).

\subsection{Representations of the algebra ${\cal S}$.}

In this subsection we study the representations of the algebra ${\cal S}$.

One can check that one-dimensional ${\cal S}$ - module is a one-dimensional module over ${\cal R} = {\cal A} * {\cal B}$. These modules are the modules over ${\cal R}^{ab} = k^{\oplus 3} \otimes k^{\oplus 3}$. Thus, there are nine one-dimensional ${\cal S}$-modules.
One can check that there are nine one-dimensional $S_x$ - modules for any $x \in {\mathbb P}^3$.

Further, consider a subspace $V \subset \overline{A} \oplus \overline{B}$ in the general position to $\overline{A}$ and $\overline{B}$. 
Using proposition \ref{comm}, we get that ${\rm mid}({\cal S} \otimes_{\mathbb F} \overline{\mathbb F}) \le 3$, where $\overline{\mathbb F}$ is an algebraic completion of ${\mathbb F}$.

We will construct a character $\chi$ of ${\cal R}(V)$ such that ${\cal R}$-module $M = {\cal R} \otimes_{{\cal R}(V)} \overline{\mathbb F}^{\chi}$ has a structure of ${\cal S}$ - module.

Denote by $t_1$ and $t_2$ the elements $p_1 - y_2q_1 - y_3q_2$ and $p_2 + q_1 + y_1q_2$.
It is easy that $p_1 = t_1 + y_2q_1 + y_3q_2$ and $p_2 = t_2 - q_1 - y_1q_2$. Recall that $p_i$ are the orthogonal idempotents. Thus, we get
the following relations between $t_1,t_2$ and $q_1,q_2$:
\begin{itemize}
\item{$(t_1 + y_2q_1 + y_3q_2)(t_1 + y_2q_1+y_2q_2) = t_1 + y_2q_1+y_3q_2$. This relation is equivalent to $p_1^2 = p_1$}
\item{$(t_2 - q_1 - y_1q_2)(t_2 - q_1-y_2q_2) = t_2 - q_1-y_2q_2$.  ($p^2_2 = p_2$)}
\item{$(t_1 + y_2q_1 + y_3q_2)(t_2 - q_1-y_1q_2) = 0$. ($p_1p_2 = 0$)}
\item{$(t_2 - q_1-y_1q_2)(t_1 + y_2q_1+y_3q_2) = 0$. ($p_2p_1 = 0$)}
\end{itemize}

These relation can be written as follows. For simplicity, denote by $d$ the element $y_3-y_1y_2 \in {\mathbb F}$. We get the following formulas:
$$
t_1q_1 = \frac{1}{d}(y_3t_1t_2+y_2y_3q_1t_2+y^2_3q_2t_2+y^2_2y_1q_1-y_2y_3q_1+y_1t^2_1+
$$
$$
+y_1y_2q_1t_1 +y_1y_3q_2t_1 - y_1t_1-y_1y_3q_2-y_1y_2q_1),
$$
$$
t_1q_2 = -\frac{1}{d}(t^2_1+y_2q_1t_1+y_3q_2t_1+y_2t_1t_2+y^2_2q_1t_2+y_2y_3q_2t_2-t_1-y_1y_2y_3q_2-y_3q_2-y_2q_1),
$$
$$
t_2q_1 = -\frac{1}{d}(-y_1y_3q_2+y_1y_2q_1-y_3t^2_2+y_3q_1t_2+y_3t_2-2y_3q_1-y_1t_2t_1+y^2_1q_2t_1+y_1y_3q_2t_2)
$$
and
$$
t_2q_2 = \frac{1}{d}(-t_2t_1+q_1t_1+y_1q_2t_1-y_2t^2_2+y_2q_1t_2+y_1y_2q_2t_2+y_2t_2-
$$
$$
-y_2q_1+y_1y_3q_2-y_1y_2q_2-y_1y_2q_2).
$$

Using the formulas for $t_i q_j, i,j = 1,2$ we get the formulas for ${\cal R}$-module ${\cal R} \otimes_{{\cal R}(V)} V$, where $V$ is a ${\cal R}(V)$-module.

Let $\chi$ be a character of ${\cal R}(V)$ given by the formula: $\chi: t_1 \mapsto z_1, t_2 \mapsto z_2$.
Denote by $v$ the generator of one-dimensional ${\cal R}(V)$-module $\overline{\mathbb F}^{\chi}$.
It is clear to see that $1 \otimes v, q_1 \otimes v, q_2 \otimes v$ are the generators of ${\cal R}$-module ${\cal R} \otimes_{{\cal R}(V)} {\mathbb K}^{\chi}$. Denote by $\rho$ the representation corresponding to ${\cal R}$ - module ${\cal R} \otimes_{{\cal R}(V)} \overline{\mathbb F}^{\chi}$
Using the formulas for $t_i q_j$, we get the following matrices:
$$
\rho(t_1) = 
\begin{pmatrix}
z_1 & \frac{1}{d}(y_3z_1z_2+y_1z^2_1-y_1z_1) & \frac{1}{d}(-z^2_1-y_2z_1z_2+z_1) \\
0 & \frac{1}{d}(y_2y_3z_2+y_1y^2_2-y_2y_3+y_1y_2z_1-y_2y_1) & \frac{1}{d}(-y_2z_1-y^2_2z_2+y_2) \\
0 & \frac{1}{d}(y_1y_3z_1+y^2_3z_2-y_3y_1) & \frac{1}{d}(-y_3z_1-y_2y_3z_2+y_1y_2y_3+y_3-y^2_3)
\end{pmatrix},
$$
$$
\rho(t_2) = 
\begin{pmatrix}
z_2 & \frac{1}{d}(y_1z_1z_2+y_3z^2_2-y_3z_2) & \frac{1}{d}(-z_1z_2-y_2z^2_2+y_2z_2)\\
0 & \frac{1}{d}(-y_1y_2-y_3z_2+2y_3-y_1z_1) & \frac{1}{d}(z_1+y_2z_2-y_2)\\
0 & \frac{1}{d}(-y_1y_3z_2+y_1y_3-y^2_1z_1) & \frac{1}{d}(y_1z_1+y_1y_2z_2+y_1y_3-y^2_1y_2-y_1y_2)
\end{pmatrix},
$$
$$
\rho(q_1) = 
\begin{pmatrix}
0 & 0 & 0\\
1 & 1 & 0\\
0 & 0 & 0
\end{pmatrix},
\rho(q_2) =
\begin{pmatrix}
0 & 0 & 0\\
0 & 0 & 0\\
1 & 0 & 1
\end{pmatrix}.
$$
One can check that the relations on $t_1,t_2$ and $q_1,q_2$ are satisfied.
The relation $t_1t_2 = t_2t_1$ is equivalent to the relations:
\begin{equation}
\label{cf}
c_1: = (y_1y_2+y_3)z_1z_2+y_2y_3z^2_2+y_1z^2_1-y_2y_3z_2-y_1z_1 = 0    
\end{equation}
\begin{equation}
\label{cs}
c_2: = z^2_1+y^2_2z^2_2+2y_2z_1z_2+(-y_1y_2-1+y_3)z_1 +
\end{equation}
$$
+ (-y_1y^2_2-y^2_2+y_2y_3)z_2 + y^2_2+y_1y_2-y_2-y_2y_3 = 0   
$$
and
\begin{equation}
\label{ct}
c_3: = y^2_1z^2_1+y^2_3z^2_2+2y_1y_3z_1z_2+(y^2_1y_2-y^2_1-y_1y_3)z_1+
\end{equation}
$$
+(y_1y_2y_3-2y^2_3)z_2+y_1y_3-y^2_1y_3-y_1y_2y_3+y_1y^2_3 = 0.
$$

These three equations define the intersection of three conics which we denote by big letters $C_1$, $C_2$ and $C_3$ in ${\mathbb F}^2$ respectively. Let us prove that these conics intersect in three points.
Consider ${\mathbb P}^2$, the natural compactification of ${\mathbb F}^2$. In this case, three conics generate three-dimensional space in the space of all conics ${\rm H}^0({\mathbb P}^2,{\cal O}(2)) = {\mathbb F}^6$. All three-dimensional subspaces in the space are parameterized by the 9-dimensional variety ${\rm Gr}(3,6)$.

{\it A determinantal curve} of three-dimensional subspaces of conics is the projective curve defined by the equation: $\{(a_1:a_2:a_3) \in {\mathbb P}^2_{(a_1:a_2:a_3)}| \quad {\rm det}(a_1c_1+a_2c_2+a_3c_3) = 0\}$, where ${c_1,c_2,c_3}$ is a basis of this subspace. Geometric properties of the subspace of conics are determined by the geometric properties of the determinant curve. One can show that three-dimensional space of conics has three common points if and only if the determinant curve is a union of three lines.

Using computer calculation on Maple, 
one can check that the determinant curve of the space generated by $c_1,c_2,c_3$ given by the formulas (\ref{cf}),(\ref{cs}) and (\ref{ct}) is a union of lines. Thus, these conics pass through three common points. 
Using these calculation, we get the cubic polynomial $f$ over ${\mathbb F}$ such that three points are defined over ${\mathbb K} = {\mathbb F}(f)$. 
Thus, we have proved the following proposition: 

\begin{Proposition}
There is an extension ${\mathbb K} \supset {\mathbb F}$ of degree 3 such that 
$|C_1 \cap C_2 \cap C_3| = 3$ in the two-dimensional affine space ${\mathbb K}^2$ with coordinates $z_1,z_2$. 
\end{Proposition}

Thus, we get that there are three points $(z^{(i)}_1,z^{(i)}_2), i = 1,2,3$ in $C_1 \cap C_2 \cap C_3 \subset {\mathbb K}^2$. Denote by ${\mathbb K}(i), i = 1,2,3$ the one-dimensional ${\cal R}(V)$ - modules corresponding to the points $(z^{(i)}_1,z^{(i)}_2) \in {\mathbb K}^2, i = 1,2,3$ respectively.
Consider ${\cal R}^{\mathbb K} = {\cal R} \otimes_{\mathbb F} {\mathbb K}$ - module $M = {\cal R} \otimes_{{\cal R}(V)} {\mathbb K}(1)$. 
Denote by $\rho$ the representation $\rho$ of algebra ${\cal R}^{\mathbb K}$ corresponding to module $M$.
Since $(z^{(i)}_1,z^{(i)}_2) \in C_1 \cap C_2 \cap C_3$, we get that $\rho(t_1t_2 - t_2t_1) = 0$. Thus, we have a structure of three-dimensional ${\cal S}^{\mathbb K} = {\cal S} \otimes_{\mathbb F} {\mathbb K}$-module on $M$.

\begin{Proposition}

${\cal S}^{\mathbb K}$ - module $M$ is irreducible.
\end{Proposition}
\begin{proof}
Straightforward.
\end{proof}

Using corollary \ref{esc}, we get the following
\begin{Theorem}
We have the following isomorphism of algebras:
\begin{equation}
{\cal S}^{\mathbb K} \cong {\mathbb K}^{\oplus 9} \oplus M_3({\mathbb K}).
\end{equation}
Thus, algebra $S_x$ for generic $x \in {\mathbb P}^3$ is isomorphic to $k^{\oplus 9} \oplus M_3(k)$.
\end{Theorem}
\begin{proof}
There are nine one-dimensional irreducible representations and one three-dimensional irreducible representation of ${\cal S}^{\mathbb K}$. Thus, ${\rm dim}_{\mathbb K} {\cal S}^{\mathbb K} = 18$ and ${\cal S}^{\mathbb K} \cong {\mathbb K}^{\oplus 9} \oplus M_3({\mathbb K})$. The rest is easy.
\end{proof}

\end{document}